\documentclass{article}
\usepackage[utf8]{inputenc}
\usepackage{comment}
\usepackage{amssymb,amsmath,amsthm,graphicx,xcolor,mathtools,enumerate,mathrsfs}
\usepackage{fullpage}
\usepackage{hyperref}

\usepackage{cleveref}
\usepackage[shortlabels]{enumitem}
\setlist[enumerate,1]{label={\upshape (\roman*)}}

\theoremstyle{plain}
\newtheorem{theorem}{Theorem}[section]
\crefname{theorem}{Theorem}{Theorems}

\newtheorem{proposition}[theorem]{Proposition}
\crefname{proposition}{Proposition}{Propositions}

\crefname{corollary}{Corollary}{Corollaries}

\newtheorem{lemma}[theorem]{Lemma}
\crefname{lemma}{Lemma}{Lemmas}

\crefname{conjecture}{Conjecture}{Conjectures}

\crefname{problem}{Problem}{Problem}

\newtheorem{claim}[theorem]{Claim}
\crefname{claim}{Claim}{Claims}

\crefname{observation}{Observation}{Observations}

\crefname{setup}{Setup}{Setups}

\crefname{fact}{Fact}{Facts}

\crefname{algorithm}{Algorithm}{Algorithms}

\crefname{remark}{Remark}{Remarks}

\crefname{example}{Example}{Examples}

\theoremstyle{definition}
\newtheorem{definition}[theorem]{Definition}
\crefname{definition}{Definition}{Definitions}

\crefname{construction}{Construction}{Constructions}

\crefname{question}{Question}{Questions}

\numberwithin{equation}{section}

\global\long\def\E{\mathbb{E}}
\global\long\def\P{\mathbb{P}}
\def\eps{\varepsilon}

\renewcommand{\int}[1]{\mathop{\mkern 0mu\mathrm{int}}\nolimits(#1)}

\usepackage{tikz}
\usetikzlibrary{math}
\usetikzlibrary{calc,decorations.pathmorphing,through}
\tikzset{snake it/.style={decorate, decoration=snake}}

\definecolor{DarkDesaturatedBlue}{HTML}{3A3556}
\definecolor{VividOrange}{HTML}{F15918}
\definecolor{PureOrange}{HTML}{FFBA00}
\definecolor{LightGrayishPink}{HTML}{EEC5D5}
\definecolor{VerySoftBlue}{HTML}{B5AFDB}

\begin{document}

\title{Counting spanning subgraphs in dense hypergraphs}

\author{Richard Montgomery\thanks{Supported by the European
Research Council (ERC) under the European Union Horizon 2020 research and innovation programme (grant agreement No. 947978) and the Leverhulme Trust. Mathematics Institute, Zeeman Building, University of Warwick, Coventry CV4 7AL, UK. \{richard.montgomery$|$matias.pavez-signe\}@warwick.ac.uk.} \and Mat\'ias Pavez-Sign\'e$^\ast$}
\date{}\maketitle

\begin{abstract} We give a simple method to estimate the number of distinct copies of some classes of spanning subgraphs in hypergraphs with high minimum degree. In particular, for each $k\geq 2$ and $1\leq \ell\leq k-1$, we show that every $k$-graph on $n$ vertices with minimum codegree at least
\def\arraystretch{1.5}
\[
\left\{\begin{array}{ll}
\left(\dfrac{1}{2}+o(1)\right)n & \text{ if }(k-\ell)\mid k,\vspace{.2cm}\\
\left(\dfrac{1}{\lceil \frac{k}{k-\ell}\rceil(k-\ell)}+o(1)\right)n & \text{ if }(k-\ell)\nmid k,  
\end{array}
\right.
\]
contains $\exp(n\log n-\Theta(n))$ Hamilton $\ell$-cycles as long as $(k-\ell)\mid n$. When $(k-\ell)\mid k$ this gives a simple proof of a result of Glock, Gould, Joos, K\"uhn and Osthus, while, when $(k-\ell)\nmid k$ this gives a weaker count than that given by Ferber, Hardiman and Mond or, when $\ell<k/2$, by Ferber, Krivelevich and Sudakov, but one that holds for an asymptotically optimal minimum codegree bound.
\end{abstract}

\section{Introduction}
A central problem in extremal graph theory is to find sufficient degree conditions which force the containment of a given spanning subgraph, and here a classical result of Dirac~\cite{Dirac} from 1952 states that every graph on $n\ge 3$ vertices with minimum degree at least $n/2$ contains a Hamilton cycle.
In 1995, Bollob\'as~\cite{bollobas1995} and Bondy~\cite{bondy1995basic} asked for estimates of the number of distinct Hamilton cycles in graphs satisfying Dirac's condition. S\'ark\"ozy, Selkow, and Szemer\'edi~\cite{SARKOZY2003237} used the {regularity method} in 2003 to show that every graph with $n\geq 3$ vertices and minimum degree at least $n/2$ contains at least $c^nn!$ Hamilton cycles, for some constant $c>0$. This cannot be improved to any $c>1/2$, as for fixed $p>1/2$ the typical random graph $G(n,p)$ satisfies Dirac's condition and has $(1-o(1))^np^nn!$ distinct Hamilton cycles (see~\cite{janson_1994}). In 2009, Cuckler and Kahn~\cite{cuckler2009hamiltonian} obtained precise estimates of the number of distinct Hamilton cycles in terms of the minimum degree of the host graph.
In particular, they showed that every $n$-vertex graph with minimum degree at least $n/2$ contains at least $(\frac 12-o(1))^nn!$ distinct Hamilton cycles, thus matching the bound given by random graphs.

Our purpose here is to introduce a simple method to bound below the number of copies of many different spanning subgraphs in graphs and hypergraphs with high minimum degree, and apply this to give new counting results in dense hypergraphs (Theorems~\ref{mainthm} to~\ref{thm:factors}). However, let us already illustrate this method as it would apply to Hamilton cycles in an $n$-vertex graph $G$ with minimum degree $\delta(G)\geq (\frac{1}{2}+\eps)n$ with $\eps>0$ fixed and $n$ large. Set $r=\mu n$ with $1/n\ll \mu \ll \eps$, and partition $V(G)=V_1\cup \ldots \cup V_r$ by choosing the location of each vertex independently and uniformly at random. With probability at least $e^{-n}$, the minimum degree of each subgraph $G[V_i]$ will be at least $(\frac{1}{2}+\frac{\eps}{2})|V_i|$ and there will be a disjoint collection of $r$ edges in $G$ connecting the sets $V_1,\ldots,V_r$ in a cycle in order and connecting $V_r$ to $V_1$ (see Figure~\ref{fig:new}). Applying classical methods to find a Hamilton path through each subgraph $G[V_i]$ to connect these edges creates a Hamilton cycle of $G$ which passes through each of the sets $V_1,\ldots,V_r$ in order. Even though the probability of success here is small, only at least $e^{-n}$, for this to be true it is easy to show that $G$ must contain at least $c^nn!$ distinct Hamilton cycles (for some fixed $0<c\ll \varepsilon$).

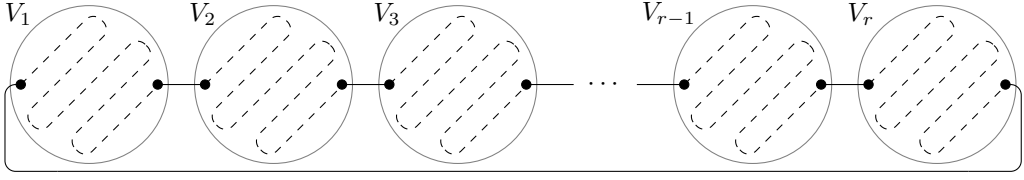
\begin{figure}[t]
    \centering
    \begin{tikzpicture}[scale=.7]
\def\spacer{3.5}
\def\circrad{1.5};
\def\innerrad{1.3};
\def\vxrad{2.5pt};

\foreach \n in {1,2,3}
{
\coordinate (A\n) at ($(\n*\spacer,0)$);
\coordinate (B\n) at ($(A\n)-(\innerrad,0)$);
\coordinate (C\n) at ($(A\n)+(\innerrad,0)$);
}
\foreach \n in {5,6}
{
\coordinate (A\n) at ($(\n*\spacer-0.4*\spacer,0)$);
\coordinate (B\n) at ($(A\n)-(\innerrad,0)$);
\coordinate (C\n) at ($(A\n)+(\innerrad,0)$);
}
\coordinate (B4) at ($(A3)+(\spacer,0)-(\innerrad,0)$);
\coordinate (C4) at ($(A5)-(\spacer,0)+(\innerrad,0)$);
\coordinate (A4) at ($0.5*(B4)+0.5*(C4)$);

\foreach \n in {1,2,3,5,6}
{
\draw [black!50] (A\n) circle[line width =2pt,radius=\circrad];
}

\foreach \n in {1,2,3}
{
\draw ($(A\n)+(135:1.9)$) node {$V_{\n}$};
}
\draw ($(A5)+(135:1.9)-(0.25,0)$) node {$V_{r-1}$};
\draw ($(A6)+(135:1.9)-(0.1,0)$) node {$V_{r}$};
\draw (A4) node {$\cdots$};

\def\nudge{0.232};
\foreach \n in {1,2,3,5,6}
{
\draw [fill] (B\n) circle [radius=\vxrad];
\draw [fill] (C\n) circle [radius=\vxrad];
\draw [rounded corners,dashed] (B\n) -- ++(45:1.5*\innerrad) -- ++(315:\nudge);
\draw [rounded corners,dashed] ($(B\n)+(45:1.5*\innerrad)+(315:\nudge)$) -- ++(315:\nudge)-- ++(225:1.8*\innerrad) -- ++(315:\nudge);
\draw [rounded corners,dashed] ($(B\n)+(45:-0.3*\innerrad)+(315:3*\nudge)$) -- ++(315:\nudge)-- ++(45:2*\innerrad) -- ++(315:\nudge);
\draw [rounded corners,dashed] ($(B\n)+(45:1.7*\innerrad)+(315:5*\nudge)$) -- ++(315:\nudge)-- ++(225:1.8*\innerrad) -- ++(315:\nudge);
\draw [rounded corners,dashed] ($(B\n)+(45:-0.1*\innerrad)+(315:7*\nudge)$) -- ++(315:\nudge)-- ++(45:1.55*\innerrad);
}

\foreach \x/\y in {1/2,2/3,3/4,4/5,5/6}
{
\draw (C\x)--(B\y);
}
\draw [rounded corners] (B1) -- ++(-0.3,0) -- ++(0,-1.1*\circrad) -- ++ (1,0);
\draw [rounded corners] (C6) -- ++(0.3,0) -- ++(0,-1.1*\circrad) -- ++ (-1,0);
\draw ($(B1)+(-0.3,0)+(0,-1.1*\circrad)+(1,0)$) -- ($(C6) +(0.3,0)+(0,-1.1*\circrad)+ (-1,0)$);
        \end{tikzpicture}
        \caption{A Hamilton cycle passing through the sets $V_1,V_2,\ldots,V_r$ in order.}\label{fig:new}
\end{figure}

For each $i\in [r]$ in the above argument we should expect $\delta(G[V_i])\ge (\frac{1}{2}+\frac{\eps}{2})|V_i|$ with some constant probability close to 1 (as $\mu\ll \eps$), and so we would expect this to hold for all $i\in [r]$ with probability at least $2^{-r}\geq e^{-n}$ (say) for $n$ large. Because of the dependencies here, this is not straightforward to prove, but we do this with an iterative partitioning argument inspired by a technical aspect of the iterative absorption techniques introduced by Barber, Lo, K\"uhn and Osthus~\cite{barber2016edge}. The result of this argument in the graph case for Hamilton cycles is much weaker than what is already known, but this argument can be used easily in hypergraphs if the subgraph sought can be constructed from pieces like the cycle in Figure~\ref{fig:new}. This allows counting results to be inferred from the extremal minimal degree for these pieces applied to each hypergraph induced on the sets $V_i$ in the partition (with some modification to make the required connections).

Dirac's theorem has been generalised to give minimum degree conditions implying the containment of many other spanning subgraphs, including $F$-factors~\cite{hajnal1970proof,kuhn2009minimum}, trees with bounded degree~\cite{CLNS10,KSS95,KSS2001}, powers of Hamilton cycles~\cite{KSS98,KSS1998b}, and, more generally, graphs with bounded degree and sublinear bandwith~\cite{BST2009} (see also the excellent surveys~\cite{kuhn2009embedding,simonovits2019embedding}). For hypergraphs much less is known, but we will recall the progress made for Hamilton $\ell$-cycles, powers of tight cycles and factors (in Section~\ref{sec:diractype}), before discussing previous counting results and our main theorems (in Section~\ref{sec:counting}). Our technique may be applicable to other spanning subgraphs, and in particular we note that recent work of Gupta, Hamann, M{\"u}yesser, Parczyk and Sgueglia~\cite{gupta2022general} classifies some hypergraphs our techniques may apply to.


\subsection{Dirac-type problems in hypergraphs}\label{sec:diractype}
 A $k$-uniform hypergraph (or $k$-graph) is a hypergraph where every edge consists of exactly $k$ vertices. For a $k$-graph $H$ and a subset of vertices $S\subset V(H)$, the degree of $S$, denoted $d_H(S)$, is the number of edges in $H$ containing $S$. For $1\le d\le k-1$, the minimum $d$-degree of $H$, denoted $\delta_d(H)$, is the minimum of $d_H(S)$ over all subsets $S\subseteq V(H)$ with $|S|=d$. When $d=k-1$ this is the \textit{minimum codegree}, $\delta(H)=\delta_{k-1}(H)$.

The first subgraphs we consider in hypergraphs are the Hamilton \textit{$\ell$-cycles}, a well-studied generalisation of Hamilton cycles to hypergraphs.
For $k\ge 2$ and $0\le \ell\le k-1$, say that a $k$-graph $C$ is an {$\ell$-cycle} if there is a cyclic ordering $v_1,\ldots, v_t$ of $V(C)$ such that every edge of $C$ consists of $k$ consecutive vertices and every two consecutive edges intersect in exactly $\ell$ vertices. If $\ell=k-1$, then $C$ is called a \textit{tight cycle}, and, if $\ell=0$, then $C$ is a matching. A $k$-graph $H$ contains a \textit{Hamilton $\ell$-cycle} if there is an $\ell$-cycle $C\subseteq H$ with $V(C)=V(H)$ (where it can only exist if $k-\ell$ divides $|H|$).

The asymptotic minimum codegree required to guarantee a Hamilton $\ell$-cycle in an $n$-vertex $k$-graph is known due to R\"odl, Ruci\'nski, and Szemer\'edi~\cite{rodl2008approximate} if $(k-\ell)|\ell$ and to K\"uhn, Mycroft, and Osthus~\cite{kuhn2010hamilton} if $(k-\ell)\nmid k$, whose results together give the following theorem.

\begin{theorem}\label{theorem:Dirac} For each $\gamma>0$, $k\ge 2$, and $1\leq \ell<k$, there exists $n_0$ such that the following holds for all $n\ge n_0$ with $(k-\ell)\mid n$. If $H$ is an $n$-vertex $k$-graph with $\delta(H)\ge (\delta_{k,\ell}+\gamma)n$, where
\begin{equation}\label{eq:deltakl}
\delta_{k,\ell}:=\left\{\begin{array}{ll}
\dfrac{1}{2} & \text{ if }(k-\ell)\mid k, \\
\dfrac{1}{\lceil\frac{k}{k-\ell}\rceil(k-\ell)} & \text{ if }(k-\ell)\nmid k,
\end{array}
\right.
\end{equation} then $H$ contains a Hamilton $\ell$-cycle.
\end{theorem}
Due to Theorem~\ref{theorem:Dirac}, we say the minimum codegree threshold for a $k$-graph to contain a Hamilton $\ell$-cycle is $\delta_{k,\ell}$.
Note that a Hamilton tight cycle contains a Hamilton $\ell$-cycle for every $\ell$ with $(k-\ell)\mid k$. The results in Theorem~\ref{theorem:Dirac} are tight in every case up to the `error term' of $\gamma n$. For more discussion of this, and the seemingly much more difficult problem for other degree bounds $\delta_d(H)$, $d<k-1$, see the survey by K\"uhn and Osthus~\cite{kuhn2014hamilton}.

We will also consider the powers of Hamilton tight cycles, where, for each $t\ge k\ge 2$, a $k$-graph $C$ is the $(t-k+1)$th power of a tight cycle if there is a cyclic ordering $v_1,\ldots,v_s$ of $V(C)$ so that $\{v_i,\ldots,v_{i+t-1}\}$ spans a $k$-uniform clique for all $i\in [s]$ (working modulo $s$). When  $k=2$, this coincides with the usual definition of powers of cycles in graphs, where for each $t\geq 2$ the minimum degree threshold for the containment of the $(t-1)$th power of a Hamilton cycle was famously shown to be $\frac{t-1}{t}$ by Koml\'os, S\'ark\"ozy, and Szemer\'edi~\cite{KSS98,KSS1998b}.
In 2020, Bedenknecht and Reiher~\cite{bedenknecht2020squares} proved that $3$-graphs with minimum codegree at least $(4/5+o(1))n$ contain the square of a tight Hamilton cycle (which corresponds to $t=k+1$ and $k=3$), where it is known that the constant $4/5$ cannot be reduced below $3/4$ for all $n$. This was recently widely extended by Pavez-Sign\'e, Sanhueza-Matamala and Stein~\cite{pavez2021towards}, who showed that, with $\gamma>0$ fixed, if an $n$-vertex $k$-graph $H$ has minimum codegree
\begin{equation}\label{eq:cliques}\delta(H)\ge \left(1-\frac{1}{\binom{t-1}{k-1}+\binom{t-2}{k-2}}+\gamma\right)n,\end{equation}
then $H$ contains the $(t-k+1)$th power of a tight Hamilton cycle, provided that $n$ is sufficiently large. It is not known whether the bounds given in \eqref{eq:cliques} are tight up to $\gamma n$, though this is true for the cases $t\geq k=2$ and $t=k\geq 2$ which were already known~\cite{rodl2008approximate,KSS98}.

Finally, we will consider factors in hypergraphs. For a $k$-graph $F$, a $k$-graph $H$ contains an $F$-factor if it contains a collection of vertex-disjoint copies of $F$ covering every vertex in $H$. Thus, a necessary condition for an $F$-factor in $H$ is that $|F|$ divides $|H|$. For $1\le d\le k-1$, then, let $\mu_{k,d}(F)$ be the smallest number such that for every $\gamma>0$, there is $n_0$ such that if $H$ is an $n$-vertex graph with $n\ge n_0$ divisible by $|F|$ and $\delta_d(H)\ge (\mu_{k,d}(F)+\gamma)\binom{n}{k-d}$, then $H$ contains an $F$-factor. In contrast to the graph case, where the threshold is known (with moreover a much stronger error term) for all fixed $F$ due to Koml\'os, S\'ark\"ozy and Szemer\'{e}di~\cite{hajnal1970proof} and K\"uhn and Osthus~\cite{kuhn2009minimum}, for most cases we do not have good bounds even in the case $d=k-1$ (see the survey by~\cite{zhao2016recent} and references therein).


\subsection{Counting spanning hypergraphs}\label{sec:counting}
As in Cuckler and Kahn's work on Hamilton cycles~\cite{cuckler2009hamiltonian}, it is reasonable to believe that an $n$-vertex $k$-graph $H$ with $\delta(H)\ge\delta n$, for $(k-\ell)\mid n$ and $\delta >\delta_{k,\ell}$, should contain at least
\begin{equation}\label{eq:counting}(1-o(1))^n\cdot \Psi_{k,\ell}(n,\delta)\end{equation}
distinct Hamilton $\ell$-cycles, where $\Psi_{k,\ell}(n,\delta)$ denotes the expected number of distinct Hamilton $\ell$-cycles in the binomial random $k$-graph on $n$ vertices with edge probability $\delta$. In 2016, Ferber, Krivelevich, and Sudakov~\cite{ferber2016counting} showed that the lower bound~\eqref{eq:counting} is correct for every $\delta>1/2$ and $1\le \ell\le k/2$, and asked if this can be extended to $1\le \ell\le k-1$ and $\delta >\delta_{k,\ell}$. This was partially answered by Glock, Gould, Joos, K\"uhn, and Osthus~\cite{glock2021counting}, who showed that for every $\delta >1/2$ and $1\le \ell \le k-1$, the number of distinct Hamilton $\ell$-cycles is $\exp(n\log n-\Theta(n))$, which is tight up to the $\Theta(n)$ error term in the exponent. This result was recently improved by Ferber, Hardiman, and Mond~\cite{ferber2021counting}, who proved that the lower bound~\eqref{eq:counting} holds for every $\delta >1/2$ and $1\le \ell\le k-2$, thus settling the problem for every $1\le \ell\le k-2$ such that $(k-\ell)\mid k$. Our contribution is to use the simple method outlined above to get a bound matching that in~\cite{glock2021counting} that holds for any $\delta>\delta_{k,\ell}$ (with $\delta_{k,\ell}$ as defined in~\eqref{eq:deltakl}), thus giving a new result when $(k-\ell)\nmid k$ and $\ell>k/2$, as follows.

\begin{theorem}\label{mainthm} For each $k\geq 2$, $1\leq \ell\leq k-1$ and $\gamma>0$, there exist $n_0$ and $C$ such that the following holds for any $n\ge n_0$ with $(k-\ell)\mid n$. Any $n$-vertex $k$-graph $H$ with $\delta(H)\geq (\delta_{k,\ell}+\gamma)n$ (see~\eqref{eq:deltakl}) contains at least $\exp(n\log n-Cn)$ distinct Hamilton $\ell$-cycles.
\end{theorem}

No previous bounds on the count of powers of Hamilton tight cycles in hypergraphs with large codegree have been shown (including in graphs), and here we use our technique to similarly get a bound tight up to $\Theta(n)$ error term in the exponent, as follows.

\begin{theorem}\label{thm:powers} For each $t\ge k\ge 2$ and $\gamma>0$, there exist $n_0$ and $C$ such that the following holds for any $n\ge n_0$. Any $n$-vertex $k$-graph $H$  satisfying~\eqref{eq:cliques} contains at least $\exp(n\log n-Cn)$ distinct copies of the $(t-k+1)$th power of a Hamilton tight cycle.
\end{theorem}

Finally, we consider counting $F$-factors in dense hypergraphs. When $F$ is a single edge and $k|n$, the number of $F$-factors in a $k$-graph $H$ with $n\geq 3k$ vertices is equal to the number of Hamilton 0-cycles multiplied by a factor of $((n/k)-1)!/2$. Thus, from the results of Ferber, Krivelevich, and Sudakov~\cite{ferber2016counting} described above, if $k\mid n$ then any $n$-vertex $k$-graph $H$ with $\delta(H)\geq \delta n$ has at least $(1-o(1))^n \Psi_{k,\ell}(n,\delta)/(n/k)!$ $F$-factors (with $\Psi_{k,\ell}(n,\delta)$ as defined in~\eqref{eq:counting}).
In each case as part of wider work, for $1\le d\le k-1$,  Kang, Kelly, K\"uhn, Osthus and Pfenninger~\cite{kang2022perfect} and Pham, Sah, Sawhney and Simkin~\cite{pham2023toolkit} showed that if $F$ is again a single edge, and $k\mid n$, then any $n$-vertex $k$-graph $H$ with $\delta_d(H)\geq(\mu_{k,d}(F)+\gamma)\binom{n}{k-d}$ contains at least $c^{-n}\exp((1-1/k)n\log n)$ $F$-factors where $c\ll\gamma$ is fixed and this result is tight up to the constant $c$.
When $d=k-1$, Kang, Kelly, K\"uhn, Osthus, and Pfenninger~\cite{kang2022perfect} even managed to remove the error term in the minimum degree condition.

In the graph case ($k=2$), Pham, Sah, Sawhney, and Simkin~\cite{pham2023toolkit} very recently showed that if $F=K_t$ is the $t$-vertex complete graph, and $t\mid n$, then any $n$-vertex graph $G$ with $\delta(G)\geq (1-1/t)n$ contains $c^{-n}\exp((1-1/t)n\log n)$ $F$-factors for some fixed constant $c$. The degree bound here is the famously tight Hajnal-Szemer\'edi bound from~\cite{hajnal1970proof} and the bound on the number of $F$-factors is tight up to the constant $c$. This proved a recent conjecture of  Allen, B\"ottcher, Corsten, Davies, Jenssen, Morris, Roberts, and Skokan~\cite{allen2022robust}.

Here, our contribution again is to apply our methods to easily match these weaker bounds of $c^{-n}\exp((1-1/|F|)n\log n)$ under the stronger approximate minimum degree condition, but to do this for $F$-factors for any fixed graph $F$ and all degree bounds, as follows.

\begin{theorem}\label{thm:factors}For each $k\ge 2$, $1\le d\le k-1$, and each $k$-graph $F$ on $t\ge k$ vertices, there exists $n_0$ and $C$ such that the following holds for any $n\ge n_0$ with $t\mid n$. Any $n$-vertex $k$-graph $H$ with $\delta_d(H)\ge (\mu_{k,d}(F)+\gamma)\binom{n}{k-d}$ contains at least $\exp\big((1-\frac{1}{t})n\log n-Cn\big)$ distinct $F$-factors.
\end{theorem}

\section{Proofs}
A $k$-graph $H$ has vertex set $V(H)$ and edge set $E(H)$, and $|H|=|V(H)|$. For any $U,S\subset V(H)$ with $|U|\le k-1$,  $d(U,S)$ is the degree of $U$ in $S$, i.e., the number of edges of $H$ containing $U$ whose vertices not in $U$ are all in $S$. The hypergraph $H[S]$ induced by $S\subset V(H)$ has vertex set $S$ and edge set consisting of all those edges in $H$ contained in $S$.
For a set $X$ and $1\le \ell \le |X|$, let $\binom{X}{\ell}$ denote the collection of subsets of $X$ of size $\ell$, and let $(X)_\ell$ denote the set of tuples $\mathbf x=(x_1,\ldots,x_\ell)\in X^\ell$ of distinct elements in $X$. We will use bold letters to denote elements from $(X)_{\ell}$ or $X^\ell$. For $a,b\in(0,1]$, we will write $a\ll b$ to denote that, given $b$, we can choose $a$ sufficiently small so that the subsequent statements hold.


\subsection{Our main partitioning lemma}
Here we prove our main lemma, showing that in any linear (in $|H|$) minimum degree hypergraph $H$ there are many partitions of $V(H)$ into sets with chosen sizes whose induced subgraphs from $H$ have high minimum degree relative to their sizes (see Lemma~\ref{lemma:partition}). We need a slightly stronger condition to connect the subgraphs then found in these induced subgraphs, which motivates the following definition of a \textit{good partition}.
\begin{definition}Let $k\ge 2$, $\delta\in [0,1]$ and $\mathbf{n}=(n_1,\dots, n_r)\in\mathbb N^r$. For an $n$-vertex $k$-graph $H$, we say that a partition $V(H)=V_1\cup \ldots\cup V_r$  is \textit{$(\mathbf{n},\delta)$-good} if (working modulo $r$ in the indices) we have
    \begin{enumerate}[label = \textbf{P\arabic{enumi}}]
        \item $|V_i|=n_i$ for each $i\in [r]$, and\label{P1}
        \item for each $i\in [r]$ and $U\subseteq V_{i-1}\cup V_i\cup V_{i+1}$ with $|U|=k-1$, $d(U,V_i)\ge \delta|V_i|$.\label{P2}
    \end{enumerate}\end{definition}

To find many good partitions, we will choose an appropriate distribution $\mathbf{n}$ for the size of the subsets and partition $V(H)=V_1\cup\ldots\cup V_r$ uniformly at random subject to \ref{P1}. We then show (for the parameters we use, and with $n$ large) that \ref{P2} holds with probability at least $e^{-n}$, a relatively small probability but still enough to show that many partitions are good. For any fixed $i\in [r]$, \ref{P2} will hold with constant probability, and the only difficulty here is to show that (despite many dependencies) this is true for all $i\in [r]$ with probability at least $e^{-n}$. To do this, we form the random partition iteratively, each time dividing the subsets in two, and tracking how the minimum degree condition (or more precisely something akin to \ref{P2}) changes in the subgraphs induced on the sets. This is inspired by part of the analysis in the work by Barber, Lo, K\"uhn and Osthus~\cite{barber2016edge} introducing iterative absorption, and we use some similar calculations to those in~\cite{barber2016edge}.

For our analysis, we need the following standard concentration result for hypergeometric random variables (see, e.g.,~\cite{JLR2000} for the standard definition of such a variable with parameters $N$, $n$ and $m$).
\begin{theorem}[see, e.g., Theorem 2.10 in~\cite{JLR2000}]\label{hypergeom}
Let $X$ be a hypergeometric random variable with parameters $N$, $n$ and $m$. Then, for any $t>0$,
\[
\P(|X-\E X|\geq t)\leq 2e^{-2t^2/n}.
\]
\end{theorem}
We can now state and prove our key lemma, as follows.

\begin{lemma}\label{lemma:partition}Let $1\le \ell<k$ and let $1/n\ll 1/m\ll\delta,\gamma,1/k$ satisfy $(k-\ell)\mid n$. Then, there exists a tuple $\mathbf{n}=(n_1,\dots,n_r)$ with $\sum_{i\in [r]}n_i=n$ and, for each $i\in [r]$, $m\le n_i\le 5m$ and $(k-\ell)\mid n_i$, such that the following holds. If $H$ is an $n$-vertex $k$-graph with $\delta(H)\ge (\delta +\gamma)n$, then the number of $(\mathbf n,\delta+\gamma/2)$-good partitions of $V(H)$ is at least $e^{-n}\binom{n}{n_1,\ldots, n_r}$.
\end{lemma}
\begin{proof} Let $s$ satisfy $2m\leq n/2^s< 4m$ and let $r=2^s$. Since $(k-\ell)\mid n$, we can choose integers $n_i$, $i\in [r]$, so that $m\leq n_i\leq 5m$ and $(k-\ell)|n_i$, for each $i\in [r]$,  $\sum_{i\in [r]}n_i=n$,
and $|n_i-n_j|\le 2k$ for all $1\le i<j\le r$. We will show that the lemma holds with $\mathbf{n}:=(n_1,\ldots,n_r)$, so let $H$ be any $n$-vertex $k$-graph with $\delta(H)\geq (\delta+\gamma)n$.

We start by iteratively partitioning $V(H)$ in 2 at random. For each $0\leq i\leq s$, let $r_i=2^i$, and, for each $0\leq i\leq s$ and $j\in [r_i]$, let
\begin{equation}\label{eqn:mij}
m_{i,j}=\sum_{i'=(j-1)\cdot 2^{s-i}+1}^{j\cdot 2^{s-i}}n_{i'}.
\end{equation}
Let $V_{0,1}=V(H)$. Iteratively, do the following for each $i\in [s]$. For each $j\in [r_{i-1}]$, uniformly at random divide $V_{i-1,j}$ into two sets $V_{i,2j-1}$ and $V_{i,2j}$ so that $|V_{i,2j-1}|=m_{i,2j-1}$ and $|V_{i,2j}|=m_{i,2j}$, noting that
\begin{align*}
|V_{i-1,j}|&=m_{i-1,j}=\sum_{i'=(j-1)\cdot 2^{s-i+1}+1}^{j\cdot 2^{s-i+1}}n_{i'}
=\sum_{i'=(2j-2)\cdot 2^{s-i}+1}^{(2j-1)\cdot 2^{s-i+1}}n_{i'}+\sum_{i'=(2j-1)\cdot 2^{s-i}+1}^{2j\cdot 2^{s-i+1}}n_{i'}=m_{i,{2j-1}}+m_{i,2j}.
\end{align*}
Note that this process ends with the partition $V(H)=V_{s,1}\cup V_{s,2}\cup \ldots\cup V_{s,r}$ with $|V_{s,i}|=m_{s,i}=n_i$ for each $i\in [r]$, whose distribution is that of a partition of $V(H)$ chosen uniformly at random subject to these set sizes.

Now, for each $0\leq i\leq s$ and $j\in [r_{i}]$, let $E_{i,j}$ be the event where, for every $U\subset V_{i,j-1}\cup V_{i,j}\cup V_{i,j+1}$ (with, as in later occurrences,  addition modulo $r_i$ in the second subscript), if $|U|=k-1$, then
  \begin{equation}\label{eq:degree}
  d(U,V_{i,j})\geq \Big(\delta+\gamma-2m_{i,j}^{-1/4}\Big)m_{i,j}.
  \end{equation}
For each $0\leq i\leq s$, let $E_i$ be the event that $E_{i,j}$ holds for all $j\in [r_i]$, noting that $E_0$ holds because $\delta(H)\geq (\delta+\gamma)n$.
We will now show that the lemma is implied by the following claim.
\begin{claim}For each $i\in [s]$, $\P(E_i|E_{i-1})\geq \exp(-r_{i-1})$.\label{claim}\end{claim}
Note that if $E_s$ holds, then $(V_{s,1},\ldots,V_{s,r})$ is  $(\mathbf{n}, \delta+\gamma/2)$-good as $m_{s,i}=n_i\geq m$ for each $i\in [r]$ and $1/m\ll \gamma$. Thus, considering the distribution of the random partition $V(H)=V_{s,1}\cup V_{s,2}\cup \ldots\cup V_{s,r}$,
the number of $(\mathbf{n}, \delta+\gamma/2)$-good partitions of $V(H)$ is at least $\P(E_s)\cdot\frac{n!}{n_1!,\ldots,n_r!}$, so that the lemma follows from the claim as
\[
\P(E_s)\geq \prod_{i\in [s]}\P(E_i|E_{i-1})\geq \exp\Big(-\sum_{i\in [s]}r_{i-1}\Big)\geq \exp(-r_s)=\exp(-2^s)\geq \exp(-n/2m)\geq \exp(-n).\]
Thus, it is only left to prove the claim.

\medskip

\noindent\emph{Proof of Claim~\ref{claim}.}
Fix $i\in [s]$. For each $j\in [r_{i-1}]$, let $F_{i-1,j}$ be the event that, for every $U\subset V_{i-1,j-1}\cup V_{i-1,j}\cup V_{i-1,j+1}$ with $|U|=k-1$, we have
  \[
  d(U,V_{i,2j-1})\geq (\delta+\gamma-2m_{i,2j-1}^{-1/3})m_{i,2j-1}\;\;\text{ and }\;\;d(U,V_{i,2j})\geq (\delta+\gamma-2m_{i,2j}^{-1/3})m_{i,2j}.
  \]
 Note that, for each $j\in [r_{i-1}]$, as $V_{i-1,j}=V_{i,2j-1}\cup V_{i,2j}$, if $F_{i-1,j}$ holds, then both $E_{i,2j-1}$ and $E_{i,2j}$ hold. Furthermore, once we have chosen the partition $V(H)=V_{i-1,1}\cup V_{i-1,2}\cup \ldots\cup V_{i-1,r_{i-1}}$, the events $F_{i-1,j}$, $j\in [r_{i-1}]$, are independent. We will show that, if we choose a partition $V(H)=V_{i-1,1}\cup V_{i-1,2}\cup \ldots\cup V_{i-1,r_{i-1}}$ for which $E_{i-1}$ holds, then, for each $j\in [r_{i-1}]$, the probability that $F_{i-1,j}$ holds is at least $e^{-1}$, and thus the probability that every such $F_{i-1,j}$, $j\in [r_{i-1}]$, and hence $E_{i}$, holds, is at least $\exp(-r_{i-1})$. If this holds for every partition $V(H)=V_{i-1,1}\cup V_{i-1,2}\cup \ldots\cup V_{i-1,r_{i-1}}$ for which $E_{i-1}$ holds, we then have that 
$\P(E_i|E_{i-1})\geq\exp(-r_{i-1})$,
as required. 

Suppose then that we have chosen our partition $V(H)=V_{i-1,1}\cup V_{i-1,2}\cup \ldots\cup V_{i-1,r_{i-1}}$, and that $E_{i-1}$ holds, and let $j\in [r_{i-1}]$. 
Let $U\subset V_{i-1,j-1}\cup V_{i-1,j}\cup V_{i-1,j+1}$ satisfy $|U|=k-1$.
Note that $d(U,V_{i,2j-1})$ has a hypergeometric distribution with parameters $N':=m_{i-1,j}$, $n':=m_{i,2j-1}$ and $m':=d(U,V_{i-1,j})$. Furthermore, as $E_{i-1}$, and hence $E_{i-1,j}$ holds, we have that
\[\E [d(U,V_{i,2j-1})]=\frac{m_{i,2j-1}}{m_{i-1,j}}\cdot d(U,V_{i-1,j})\overset{\eqref{eq:degree}}{\ge} \Big(\delta+\gamma -2m_{i-1,j}^{-1/4}\Big)m_{i,2j-1}.\]
Therefore, by Theorem~\ref{hypergeom}, we have
\begin{equation}\label{eqn:1}
\P\Big(d(U,V_{i,2j-1})\leq \Big(\delta+\gamma-2m_{i-1,j}^{-1/4}\Big)m_{i,2j-1}-m_{i,2j-1}^{2/3}\Big)\leq 2\exp(-2m_{i,2j-1}^{1/3}).
\end{equation}
Now, for each $i',j'\in [s]$ we have $|n_{i'}-n_{j'}|\leq 2k\leq n_{i'}/10$ and hence $n_{i'}\leq 11n_{j'}/10$. Thus, by \eqref{eqn:mij} we have $m_{i,2j-1}\leq 11 m_{i-1,j}/20$ and hence
\begin{equation}\label{eqn:2}
2m_{i-1,j}^{-1/4}\cdot m_{i,2j-1}+m_{i,2j-1}^{2/3}\leq (\tfrac{11}{20})^{1/4}\cdot 2 m_{i,2j-1}^{3/4}+m_{i,2j-1}^{2/3}\leq 2m_{i,2j-1}^{3/4},
\end{equation}
as $m_{i,2j-1}\geq m$ and $1/m\ll 1$. In combination, \eqref{eqn:1} and~\eqref{eqn:2} give us that
\[
\P\Big(d(U,V_{i,2j-1})\leq \Big(\delta+\gamma-2m_{i,2j-1}^{-1/4}\Big)m_{i,2j-1}\Big)\leq 2\exp\Big(-2m_{i,2j-1}^{1/3}\Big).
\]
Similarly, this holds with $V_{i,2j}$ and $m_{i,2j}$ in place of $V_{i,2j-1}$ and $m_{i,2j-1}$. Furthermore, from \eqref{eqn:mij} and that $n_{i'}\leq 11n_{j'}/10$ for all $i',j'\in [s]$, it follows that
$|V_{i-1,j-1}\cup V_{i-1,j}\cup V_{i-1,j+1}|\leq 66m_{i,2j-1}/10$ and $\leq 66m_{i,2j}/10$.
Therefore, using a union bound over all $U\subset V_{i-1,j-1}\cup V_{i-1,j}\cup V_{i-1,j+1}$ satisfying $|U|=k-1$, we have that $F_{i-1,j}$ holds with probability at least
\[
1-(7m_{i,2j-1})^{k-1}\cdot 2\exp(-2m_{i,2j-1}^{1/3})-(7m_{i,2j-1})^{k-1}\cdot 2\exp(-2m_{i,2j-1}^{1/3})\geq e^{-1},
\]
as required,
where we have used that $m_{i,2j-1},m_{i,2j}\geq m$ and $1/m\ll 1$. \hspace{4cm} \qed
\end{proof}


\subsection{Counting Hamilton $\ell$-cycles}
We say that a $k$-graph $P$ with $t$ vertices is an \emph{$\ell$-path}, where $1\le \ell<k$, if $(k-\ell)\mid (t-\ell)$ and there exists an ordering $v_1,\ldots, v_t$ of $V(P)$ such that every edge of $P$ consists of $k$ consecutive vertices and such that every two consecutive edges intersect in exactly $\ell$ vertices. We will usually identify an $\ell$-path $P$ with a corresponding ordering $v_1,\ldots, v_t$. The \emph{ends} of an
$\ell$-path $P=v_1\ldots v_t$ are the tuples $\mathbf{v}=(v_1,\ldots,v_\ell)$ and $\mathbf {v}'=(v_{t-\ell+1},\ldots, v_t)$, in which case we say that $\mathbf{v}$ and $\mathbf{v}'$ are $\ell$-connected by $P$. We say that a $k$-graph $H$ \emph{contains a Hamilton $\ell$-path} if there is an $\ell$-path $P$ in $H$ with $V(P)=V(H)$.

\begin{definition}
A $k$-graph $H$ is Hamilton $\ell$-path connected if, for any pair of vertex-disjoint tuples $\mathbf u,\mathbf v\in (V(H))_\ell$, there is a Hamilton $\ell$-path in $H$ which $\ell$-connects $\mathbf u$ with $\mathbf v$.
\end{definition}
The next lemma states that $k$-graphs with large minimum codegree which satisfies the natural divisibility conditions are Hamilton $\ell$-path connected. The proof of Lemma~\ref{lemma:connected} is a very straightforward modification of the original argument of K\"uhn, Mycroft and Osthus~\cite{kuhn2010hamilton} for finding Hamilton $\ell$-cycles via the \textit{absorption method} (as can be seen in the special case when $(k-\ell)\mid  k$, where this was done by Glock, Gould, Joos, K\"uhn and Osthus as~\cite[Lemma~3.7]{glock2021counting}). For completion, however, we include a proof in an appendix.

\begin{lemma}\label{lemma:connected}Let $1\leq \ell<k$ and let $1/n\ll \gamma,1/k$ satisfy $(k-\ell)\nmid (n-\ell)$. If $H$ is an $n$-vertex $k$-graph with $\delta(H)\ge (\delta_{k,\ell}+\gamma)n$ (see~\eqref{eq:deltakl}), then $H$ is Hamilton $\ell$-path connected.
\end{lemma}
Now we are ready for the proof of our first main result.
\begin{proof}[Proof of Theorem~\ref{mainthm}] Let $\delta=\delta_{k,\ell}$ and let $m$ be such that every $k$-graph on $m'\geq m/2$ vertices with minimum codegree at least $(\delta+\gamma/4)m'$ is Hamilton $\ell$-path connected (using Lemma~\ref{lemma:connected}) and such that $1/m\ll \delta,\gamma,1/k$. Let $n_0$ and $C$ be such that, for every $n\geq n_0$, $1/n\ll 1/C\ll 1/m$. Let $H$ be an $n$-vertex $k$-graph with $\delta(H)\geq (\delta+\gamma)n$.
By Lemma~\ref{lemma:partition}, there exist a tuple $\mathbf{n}=(n_1,\dots, n_r)$ with $\sum_{i\in [r]}n_i=n$ and, for each $i\in [r]$, $m\le n_i\le 5m$ and $(k-\ell)\mid n_i$, such that $V(H)$ has at least $e^{-n}\binom{n}{n_1,\ldots,n_r}$ partitions that are $(\mathbf{n},\delta+\gamma/2)$-good.

Now, given a partition $\mathcal P=(V_1,\ldots,V_r)$ of $V(H)$, say that a Hamilton $\ell$-cycle $Q$ is \emph{$\mathcal P$-respecting} if, for each $i\in [r]$ (for some direction and working modulo $r$ in the subscript), all the vertices in $V_i$ appear concurrently on $Q$, and the interval of vertices in $V_i$ on $Q$ is just before the interval of vertices in $V_{i+1}$ on $Q$.
(See Figure~\ref{fig:new} for a Hamilton cycle that is $(V_1,\ldots,V_r)$-respecting.)
\begin{claim}\label{onlyclaim}For each $(\mathbf{n},\delta+\gamma/2)$-good partition $\mathcal P=(V_1,\ldots,V_r)$, $H$ has a $\mathcal P$-respecting Hamilton $\ell$-cycle.
\end{claim}
Before proving Claim~\ref{onlyclaim}, let us show how to deduce the theorem from it. Note that any Hamilton $\ell$-cycle in $H$ respects at most $2n$ partitions $(V_1,\ldots,V_r)$ of $V(H)$ with $|V_i|=n_i$ for each $i\in [r]$, as choosing a direction and the first vertex of $V_1$ specifies the ordered partition. Then, by Claim~\ref{onlyclaim} and our bound on the number of $(\mathbf{n},\delta+\gamma/2)$-good partitions of $V(H)$, the number of Hamilton $\ell$-cycles in $H$ is at least
\begin{equation}\label{eqn:callback}
\frac{e^{-n}}{2n}\cdot\binom{n}{n_1,\ldots,n_r}=
\frac{e^{-n}}{2n}\cdot
\frac{n!}{\prod_{i\in [r]}{n_i!}}\geq \frac{e^{-n}}{2n}\cdot \frac{n!}{(5m)!^{n/m}}\geq \exp(n\log n-Cn),
\end{equation}
using Stirling's formula and that $1/n\ll 1/C\ll 1/m$.


Therefore, it is left only to prove Claim~\ref{onlyclaim}.

\medskip

\noindent\emph{Proof of Claim~\ref{onlyclaim}.} Let $\mathcal P=(V_1,\ldots,V_r)$ be an $(\mathbf{n},\delta+\gamma/2)$-good partition. For each $i\in [r]$, using that $|V_i|\geq m\geq \ell$, pick an arbitrary $\ell$-tuple $\mathbf v_i=(v_{i,1},v_{i,2},\ldots ,v_{i,\ell})\in (V_i)_{\ell}$. For each $i\in [r]$, let $H_i=H[V_i\cup \{v_{i-1,1},\ldots,v_{i-1,\ell}\}]$ (working modulo $r$ in the indices), so that, as $\mathcal P$ is $(\mathbf{n},\delta+\gamma/2)$-good, we have that $\delta(H_i)\ge (\delta+\gamma/2)|V_{i}|\geq (\delta+\gamma/4)|H_i|$. Moreover, $|H_i|-\ell=n_i$ is divisible by $k-\ell$. Therefore, by Lemma~\ref{lemma:connected}, there is a Hamilton $\ell$-path in $H_i$ with vertex sequence $v_{i-1,1}v_{i-1,2}\ldots v_{i-1,\ell}L_i v_{i,1}\ldots v_{i,\ell}$ for some sequence $L_i$. Then, the ordering
$L_1v_{1,1}\ldots v_{1,\ell}L_2v_{2,1}\ldots v_{2,\ell}L_3\ldots v_{r-1,1}\ldots v_{r-1,\ell}L_rv_{r,1}\ldots v_{r,\ell}$
is an ordering of the vertices of $H$ which gives a Hamilton $\ell$-cycle respecting the partition $(V_1,\ldots,V_r)$, as required.
\hfill \qed
\end{proof}


\subsection{Counting powers of tight Hamilton cycles}
For Theorem~\ref{thm:powers}, we use the following definitions. Let $t\ge k\ge 2$ and let $H$ be a $k$-graph. The \textit{$t$-clique graph}  of $H$, denoted $K_t(H)$, is the $t$-graph with vertex set $V(K_t(H))=V(H)$  where $\{v_1,\ldots, v_t\}$ is an edge of $K_t(H)$ if and only if $H[\{v_1,\ldots, v_t\}]$ is a $k$-uniform clique in $H$.

Our proof of Theorem~\ref{thm:powers} is very similar to the proof of Theorem~\ref{mainthm} so we do not repeat it here and only state the differences. The main thing is to note that, given
\begin{itemize}
\item  a partition $V(H)=V_1\cup \ldots \cup V_r$ and distinct vertices $v_{i,1},\ldots,v_{i,t-1}\in V_i$, $i\in [r]$, such that, for each $i\in [r]$, $H[\{v_{i,1},\ldots,v_{i,t-1}\}]$ is a $(t-1)$-clique, and
\item orderings $L_i$, $i\in [r]$, of the vertices in $V_i\setminus\{v_{i,1},\ldots,v_{i,t-1}\}$ where $v_{i-1,1},\ldots,v_{i-1,t-1}L_iv_{i,1},\ldots,v_{i,t-1}$ is the ordering of a Hamilton tight path in $K_t(H[V_i\cup \{v_{i-1,1},\ldots,v_{i-1,t-1}\}])$ for each $i\in [r]$,
\end{itemize}
the ordering $L_1v_{1,1}\ldots v_{1,t-1}L_2v_{2,1}\ldots v_{2,t-1}L_3\ldots v_{r-1,1}\ldots v_{r-1,t-1}L_rv_{r,1}\ldots v_{r,t-1}$
is an ordering of the vertices of $H$ which gives a $(t-k+1)$th power of a Hamilton tight cycle respecting the partition $(V_1,\ldots,V_r)$.
Thus, the proof of Theorem~\ref{thm:powers} follows identically to that of Theorem~\ref{mainthm} using the following proposition to select the vertices $v_{i,1},\ldots,v_{i,t-1}\in V_i$, $i\in [r]$, in place of the vertices $v_{i,1},\ldots,v_{i,\ell}\in V_i$, $i\in [r]$, and the following lemma in place of Lemma~\ref{lemma:connected}.

\begin{proposition}\label{prop:endfinder} Let $t\ge k\ge 2$ and let $1/n\ll \gamma, 1/t$. If $H$ is an $n$-vertex $k$-graph satisfying~\eqref{eq:cliques}, then $H$ contains a $t$-vertex $k$-uniform clique.
\end{proposition}

\begin{lemma}\label{lemma:connected:cliques}Let $t\ge k\ge 2$ and let $1/n\ll \gamma, 1/t$. If $H$ is an $n$-vertex $k$-graph satisfying~\eqref{eq:cliques}, then for every pair of disjoint tuples $\mathbf{v},\mathbf{v'}\in (V(H))_{t-1}$ whose vertices support a $(t-1)$-vertex $k$-uniform clique in $H$, there is a Hamilton tight path in $K_t(H)$ with ends $\mathbf{v}$ and $\mathbf{v'}$.
\end{lemma}

Proposition~\ref{prop:endfinder} can be proved simply by picking the vertices of the clique greedily, where the codegree bound used is comfortably sufficient for this. The proof of Lemma~\ref{lemma:connected:cliques} follows the same ideas as the proof of Lemma~\ref{lemma:connected}, but using results of Pavez-Sign\'e, Sanhueza-Matamala and Stein~\cite{pavez2021towards} in place of those by K\"uhn, Mycroft and Osthus~\cite{kuhn2010hamilton}. We comment further on this in the appendix.

\subsection{Counting $F$-factors}
To count factors, we no longer need to connect the different parts of a good partition, but we do wish to have conditions for different degrees than just the codegree, motivating the following definition.
\begin{definition}Let $k,d\in\mathbb N$ satisfy $1\le d\le k-1$. Let $\mu >0$ and $\mathbf{n}=(n_1,\dots, n_r)\in\mathbb N^r$. For an $n$-vertex $k$-graph $H$,  say that a partition $V(H)=V_1\cup \ldots\cup V_r$ is \textit{$(\mathbf{n},d,\mu)$-good} if, for each $i\in [r]$, $|V_i|=n_i$ and $\delta_d(H[V_i])\ge \mu \binom{n_i}{k-d}$.\end{definition}
With only minor modification, including using McDiarmid’s inequality (see Lemma~1.2 in~\cite{McDiarmid1989}) to bound similar events to the events $E_{i,j}$, the proof of Lemma~\ref{lemma:partition} can be adapted to prove the following.
\begin{lemma}\label{lemma:partition:2}Let $1\le d\le k-1$ and let $1/n\ll 1/m\ll \gamma,\mu ,1/k,1/t$ satisfy $t\mid n$. Then, there exists a tuple $(n_1,\ldots, n_r)\in\mathbb N^r$, with $\sum_{i\in [r]}n_i=n$, and $m\le n_i\le 5m$ and $t\mid n_i$ for each $i\in [r]$, such that the following holds. If $H$ is an $n$-vertex $k$-graph with $\delta_d(H)\ge(\mu+\gamma)\binom{n}{k-d}$, then the number of $(\mathbf{n},d,\mu+\gamma/2)$-good partitions in $H$ is at least $e^{-n}\binom{n}{n_1,\ldots, n_r}$.
\end{lemma}
Recalling that $\mu_{k,d}(F)$ is defined at the end of Section~\ref{sec:diractype}, we are now ready for the proof of Theorem~\ref{thm:factors}.
\begin{proof}[Proof of Theorem~\ref{thm:factors}]Let $1\le d\le k-1$ and let $F$ be a fixed $k$-graph on $t$ vertices. Let $m$ be such that every $k$-graph on $m'\ge m$ vertices, with $t\mid m'$, and minimum $d$-degree at least $(\mu_{k,d}(F)+\gamma/2)\binom{m'}{k-d}$, contains an $F$-factor
 (using the definition of $\mu_{k,d}(F)$).
Let $n_0$ and $C$ be such that, for every $n\geq n_0$, $1/n\ll 1/C\ll 1/m$, and let $H$ be an $n$-vertex $k$-graph with $\delta_d(H)\geq (\mu_{k,d}(F)+\gamma)n^{k-d}$.
By Lemma~\ref{lemma:partition:2}, there is a tuple $\mathbf{n}=(n_1,\ldots, n_r)$, with $\sum_{i\in [r]}n_i=n$, and $m\le n_i\le 5m$ and $t\mid n_i$ for each $i\in [r]$, such that the number of $(\mathbf{n},d,\mu_{k,d}(F)+\gamma/2)$-good partitions of $V(H)$ is at least $e^{-n}\binom{n}{n_1,\ldots, n_r}$.

For each $(\mathbf{n},d,\mu_{k,d}(F)+\gamma/2)$-good partition $\mathcal P=(V_1,\ldots, V_r)$, by the definition of $m$, there is an $F$-factor in $H[V_i]$ for each $i\in [r]$, and therefore $H$ has an $F$-factor, $G$ say, where $G[V_i]$ is an $F$-factor of $H[V_i]$ for each $i\in [r]$.
On the other hand, given an $F$-factor $G$, the number of possible partitions $\mathcal P=(V_1,\ldots, V_r)$, such that $G[V_i]$ is an $F$-factor of $H[V_i]$ and $|V_i|=n_i$ for each $i\in [r]$ is at most $(n/t)^{n/t}\leq \exp((n\log n)/t)$.
Therefore, the number of distinct $F$-factors in $H$ is (similarly to \eqref{eqn:callback}) at least
\[
e^{-n}\cdot \exp(-(n/t)\log n)\cdot \binom{n}{n_1,\ldots, n_r}
\ge e^{-n}\cdot \exp(-(n/t)\log n)\cdot
\frac{n!}{(5m)!^{n/m}}\geq \exp\big((1-\tfrac{1}t)n\log n-Cn\big),\]
as required, where we have used Stirling's formula and that $1/n\ll 1/C\ll 1/m$.
\end{proof}


\bibliographystyle{abbrv}
\bibliography{cc}



\appendix
\section{Proof of Lemmas~\ref{lemma:connected} and~\ref{lemma:connected:cliques}}
Here we provide a proof for Lemma~\ref{lemma:connected}. The proof is a straightforward modification of the original argument by K\"uhn, Mycroft and Osthus~\cite{kuhn2010hamilton} for finding Hamilton $\ell$-cycles in dense hypergraphs via the \textit{absorption method}. 
\begin{definition}Let $1\le \ell<k$ and let $H$ be a $k$-graph. Say that an $\ell$-path $P$ in $H$, with ends $\mathbf{a},\mathbf{b}\in (V(H))_\ell$,  can \textit{absorb} a collection of pairwise disjoint $(k-\ell)$-sets $S_1,\ldots, S_t$ if
\begin{enumerate}
    \item $P$ contains no vertex from $\bigcup_{i\in [t]}S_i$, and
    \item there is an $\ell$-path $Q$ with vertex set $V(Q)=V(P)\cup \bigcup_{i\in [t]}S_i$ and with ends $\mathbf{a}$ and $\mathbf{b}$.
\end{enumerate}
In this case, we say that $P$ is an \textit{absorbing path} for $S_1,\ldots,S_t$.\end{definition}
We may use a similar definition of absorbing paths that works for powers of tight cycles (for Lemma~\ref{lemma:connected:cliques}), in which case the absorbing paths are tight paths in the $t$-clique graph $K_t(H)$ that can absorb vertices rather than $(k-\ell)$-sets (see Definition 7.1 in~\cite{pavez2021towards}).
\begin{definition}Let $1\le \ell<k$ and let $H$ be an $n$-vertex $k$-graph. Say that a $(k-\ell)$-set $S\subset V(H)$ is \textit{$(\beta,t)$-good} if $H$ contains at least $\beta n^t$ absorbing paths for $S$, each with exactly $t$ vertices. If $S$ is not $(\beta,t)$-good, we then say that $S$ is \textit{$(\beta,t)$-bad}. \end{definition}
The following result states that most $(k-\ell)$-sets are good in $k$-graphs with linear minimum codegree. 
\begin{lemma}[Lemma~6.2 in~\cite{kuhn2010hamilton}]\label{lemma:badsets} Let $k\ge 3$ and $1\le \ell<k$ satisfy $(k-\ell)\nmid k$, and let $1/n\ll \beta \ll\theta\ll\mu ,1/k$. There exists a constant $t=t(k,\ell)$ such that if $H$ is an $n$-vertex $k$-graph with $\delta(H)\ge \mu n$, then the number of  $(\beta,t)$-bad $(k-\ell)$-sets in $H$ is at most $\theta n^{k-\ell}$.
\end{lemma}
The next lemma says that we can find a short path $P$ which can absorb any collection of $o(n)$ pairwise disjoint good sets and, moreover, every vertex outside $P$ belongs to only few bad sets.  
\begin{lemma}[Lemma 6.3 in~\cite{kuhn2010hamilton}]\label{lemma:absorbing}Let $k\geq 3$ and $1\leq \ell<k$ satisfy $(k-\ell)\nmid k$, and let $1/n\ll \alpha \ll \beta\ll \theta\ll \mu,1/k$. If $H$ is an $n$-vertex $k$-graph with $\delta(H)\geq \mu n$ and $t=t(k,\ell)$ from Lemma~\ref{lemma:badsets}, then $H$ contains an $\ell$-path $P$ on at most $\mu n$ vertices such that 
\begin{enumerate}
\item every vertex of $H-V(P)$ lies in at most $\theta n^{k-\ell-1}$ $(\beta,t)$-bad $(k-\ell)$-sets, and
\item $P$ can absorb any collection of at most $\alpha n$ disjoint $(\beta,t)$-good $(k-\ell)$-sets of vertices of $H-V(P)$.
\end{enumerate}
\end{lemma}
For proving Lemma~\ref{lemma:connected:cliques}, we can show an analogue result in the spirit of Lemma~\ref{lemma:absorbing} for the $t$-clique graph $K_t(H)$ using Lemma 7.3 in~\cite{pavez2021towards} and \textit{Step 1} in the proof of Theorem 1.1 in~\cite{pavez2021towards}. 

Next, we have a lemma for when any two disjoint ordered $\ell$-sets can be connected by a short $\ell$-path, as follows.

\begin{lemma}[Corollary 5.4 in~\cite{kuhn2010hamilton}]\label{lemma:diameter}
Let $k\geq 3$ and $1\leq \ell\leq k-1$ satisfy $(k-\ell)\nmid k$, and let $1/n\ll \mu,1/k$. If $H$ is an $n$-vertex $k$-graph with $\delta(H)\geq \mu n$, then for any two disjoint ordered $\ell$-sets $\mathbf{a},\mathbf{b}\in (V(H))_\ell$ there exists an $\ell$-path $P$ with ends $\mathbf{a}$ and $\mathbf{b}$ that contains at most $8k^5$ vertices.
\end{lemma}
We have a similar statement in the vein of Lemma~\ref{lemma:diameter} for the clique graph (see~\cite[Lemma 4.1]{pavez2021towards}), which states that any two disjoint $(t-1)$-sets that support $(t-1)$-cliques in $H$ can be connected in $K_t(H)$ by many short tight paths.

The last two ingredients we need are, firstly, that the degree conditions are preserved by taking random subsets and,  secondly, that hypergraphs with large vertex degree have perfect matchings.  
\begin{lemma}[Lemma 8.1 in~\cite{kuhn2010hamilton}]\label{lemma:reservoir}
Let $1\leq d<k$ and let $1/n\ll \alpha,\mu,1/k$. Let $H$ be an $n$-vertex $k$-graph with $\delta_d(H)\geq \mu \binom{n}{k-d}$, and let $R\subset V(H)$ be a random subset of size $\alpha n$. Then, with probability $1-o(1)$, we have  $|N_H(S)\cap \binom{R}{k-d}|\geq \mu \binom{\alpha n}{k-d}-n^{k-d-\frac{1}{3}}$ for every $S\in \binom{V(H)}{d}$.
\end{lemma}

\begin{theorem}[Perfect matching theorem~\cite{daykin1981}]\label{theorem:matching}Let $n,k\geq 2$ such that $k\mid n$. If $H$ is an $n$-vertex $k$-graph with $\delta_1(H)\ge\frac{k-1}{k}(\binom{n-1}{k-1}-1)$, then $H$ contains a perfect matching.
\end{theorem}

Now we are ready for the proof of Lemma~\ref{lemma:connected}.
\begin{proof}[Proof of Lemma~\ref{lemma:connected}]
 We begin by choosing constants
    \[1/n\ll \alpha\ll \beta\ll \theta\ll \theta'\ll \gamma \ll \mu\ll1/k,\]
and let $t=t(k,\ell)$ from Lemma~\ref{lemma:badsets}.  Given disjoint $\ell$-tuples $\mathbf{a}=(a_1,\ldots, a_\ell)$ and $\mathbf{b}=(b_1,\ldots, b_\ell)$ in $V(H)$, set $H'=H-\{a_1,\ldots,a_\ell,b_1,\ldots,b_\ell\}$ and $n'=|H'|=n-2\ell$, and note that $\delta(H')\ge (\delta_{k,\ell}+\gamma/2)n'$ as $1/n\ll 1/k$. Using Lemma~\ref{lemma:absorbing}, find an absorbing $\ell$-path $P_0$ in $H'$ with at most $\gamma n'/16$ vertices which can absorb any collection of at most $2\alpha n'$ pairwise disjoint $(\beta,t)$-good  $(k-\ell)$-sets in $V(H')$ (here we have used that $\delta(H')\ge \gamma n'
/16$). Let $G$ be an auxiliary $(k-\ell)$-graph with $V(G)=V(H')$ and edge set consisting of all those $(k-\ell)$-sets in $H'$ which are $(\beta,t)$-good. By Lemma~\ref{lemma:badsets}, for every $v\in V(G)\setminus V(P_0)$ we have
\[d_G(v)\ge \binom{n'}{k-\ell-1}-\theta n'^{k-\ell-1}\geq (1-{\theta'})\binom{n'}{k-\ell-1}.\]
Let $R\subset V(H')$ be a random subset of size $\alpha n'$. Then, by Lemma~\ref{lemma:reservoir}, with probability $1-o(1)$,
\begin{enumerate}
    \item [(i)] $d_G(v,R)\ge (1-2\theta')\binom{|R|}{k-\ell-1}$ for every $v\in V(G)\setminus V(P_0)$, and
    \item [(ii)] $d_H(S,R)\ge (\delta_{k,\ell}+\gamma/4)|R|$ for every $S\in \binom{V(H)}{k-1}$.
\end{enumerate}
Moreover, as $\E[|R\cap V(P_0)|]=\alpha |P_0|$, Lemma~\ref{hypergeom} implies that  $|R\cap V(P_0)|\le  2\alpha |P_0|\le \alpha \gamma n'/8$ with probability $1-o(1)$. Therefore, there is a choice of $R$ such that, letting $R'=R\setminus V(P_0)$, we have
\begin{enumerate}[label = \textbf{R\arabic{enumi}}]
\item\label{reservoir:1}$\alpha n'\ge |R'|\ge (1-\mu)\alpha n'$,
    \item\label{reservoir:2}$d_G(v,R')\ge (1-\mu)\binom{|R'|}{k-\ell-1}$ for every $v\in V(G)\setminus V(P_0)$, and
    \item\label{reservoir:3}$d_H(S,R')\ge (\delta_{k,\ell}+\gamma/8)|R'|$ for every $S\in \binom{V(H)}{k-1}$.
\end{enumerate}
Let $V'\subseteq V(H')\setminus (R'\cup V(P_0))$ be a subset obtained by removing at most $k-\ell$ vertices so that $|V'|$ is divisible by $k-\ell$. We write $H''=H[V']$ and note that $\delta(H'')\ge (\delta_{k,\ell}+\gamma/16)|H''|$. Use Theorem~\ref{theorem:Dirac} to find a Hamilton $\ell$-cycle in $H''$, and hence $H''$ contains an $\ell$-path $P$ with $|P|\ge|H''|-2k$. Let $\mathbf x'=(x_1',\ldots,x_\ell')$ and $\mathbf y'=(y_1',\ldots,y_\ell')$ be the ends of $P_0$, and let $\mathbf x=(x_1,\ldots,x_\ell)$ and $\mathbf y=(y_1,\ldots,y_\ell)$ be the ends of $P$. Using Lemma~\ref{lemma:diameter} and~\ref{reservoir:3}, we find sequences of vertices  $L_{\mathbf{ax'}}, L_{\mathbf{y'x}}$, and $L_{\mathbf{yb}}$ such that
\begin{itemize}
    \item $|L_{\mathbf{ax'}}|,|L_{\mathbf{y'x}}|, |L_{\mathbf{yb}}|\le 8k^5$,
    \item $L_{\mathbf{ax'}}, L_{\mathbf{y'x}}$ and $L_{\mathbf{yb}}$ are pairwise disjoint and contain vertices only from $R'$, and
    \item $P_{\mathbf{ax'}}=a_1\ldots a_\ell L_{\mathbf{ax'}}x_1'\ldots x_\ell'$,  $P_{\mathbf{y'x}}=y_1'\ldots y_\ell' L_{\mathbf{y'x}}x_1\ldots x_\ell$, and $P_{\mathbf{yb}}=y_1\ldots y_\ell L_{\mathbf{yb}} b_1\ldots b_\ell$ are $\ell$-paths.
\end{itemize}
Therefore, the sequence 
\[Q=a_1\ldots a_\ell L_{\mathbf{ax'}}P_0L_{\mathbf{y'x}}PL_{\mathbf{yb}}b_1\ldots b_\ell\]
is an $\ell$-path in $H$ connecting $\mathbf a$ with $\mathbf b$ and using at most $24k^5$ vertices from $R'$. Let $X=V(H)\setminus V(Q)$ and note that, because of \ref{reservoir:1} and \ref{reservoir:2}, every vertex $v\in X$ satisfies 
\begin{equation}\label{eq:deg:matching}d_G(v,X)\ge (1-\mu)\binom{|R'|}{k-\ell -1}-25k^5\cdot |X|^{k-\ell -2}\ge (1-2\mu)\binom{|X|}{k-\ell-1},\end{equation}
where we have used that $|X\setminus R'|\le 24k^5+2k+(k-\ell)\le 25k^5$ and $1/n\ll 1/k$. Then, by~\eqref{eq:deg:matching} and Theorem~\ref{theorem:matching}, $G[X]$ contains a perfect matching $S_1,\dots, S_j$, which is a collection of at most $|X|\le |R'|+2k+(k-\ell)\le 2\alpha n$ pairwise disjoint $(\beta,t)$-good $(k-\ell)$-sets. Finally, using Lemma~\ref{lemma:absorbing} we can absorb all the vertices from $\bigcup_{i\in [j]}S_i$ which concludes the proof.\end{proof}

The proof of Lemma~\ref{lemma:connected:cliques} follows a similar strategy. Let us sketch the main steps of the proof here. Given two disjoint tuples $\mathbf a=(a_1,\ldots, a_{t-1})$ and $\mathbf b=(b_1,\ldots, b_{t-1})$ that support $(t-1)$-cliques in $H$, set $H'=H-\{a_1,\ldots, a_{t-1},b_1,\ldots, b_{t-1}\}$. The proof of Lemma~\ref{lemma:connected:cliques} consists of the following steps:
\begin{itemize}
    \item Find a short absorbing path $P_0$ in $K_t(H')$ capable of absorbing any collection of $o(n)$ vertices.
    \item Set aside a reservoir $R$ in $H'- V(P_0)$ and find a Hamilton tight path in $K_t(H'-(V(P_0)\cup R))$ (this can be done as $H'-(V(P_0)\cup R)$ satisfies~\eqref{eq:cliques} with a slightly worse error term).
    \item Using a corresponding connecting lemma (\cite[Lemma 4.1]{pavez2021towards}) and the properties of the reservoir, find an almost spanning path $Q$ with ends $\mathbf{a}$ and $\mathbf b$ which contains $P_0$.
    \item Finish the embedding using the properties of the absorbing path.
\end{itemize}
The main difference with the proof of Lemma~\ref{lemma:connected} is that, in this case, we do not need to use any matching result in the absorption step as $H$ contains no bad vertices (see \cite[Lemma 7.2]{pavez2021towards}).
\end{document}